\documentclass{elsarticle}

\usepackage{amsmath, amsthm, amsfonts, amssymb, hyperref,tikz}
\usepackage{graphicx}

\newtheorem{theorem}{Theorem}[section]
\newtheorem{conjecture}[theorem]{Conjecture}
\newtheorem{proposition}[theorem]{Proposition}
\newtheorem{lemma}[theorem]{Lemma}
\newtheorem{corollary}[theorem]{Corollary}
\newtheorem{question}[theorem]{Question}

{\theoremstyle{definition}
\newtheorem{definition}[theorem]{Definition}
\newtheorem*{definition*}{Definition}
\newtheorem{remark}[theorem]{Remark}
\newtheorem{example}[theorem]{Example}
}

\numberwithin{equation}{section}
\numberwithin{figure}{section}

\newcommand{\bigP}[1]{\mathcal{P}_{#1}}
\newcommand{\leqs}{\leq_s}

\newcommand{\les}{<_s}
\newcommand{\leqsupp}{\leq_{\mathit{supp}}}
\newcommand{\lesupp}{<_{\mathit{supp}}}
\newcommand{\Supp}[1]{\mathit{Supp}_{#1}}
\newcommand{\eqsupp}[1]{\lfloor#1\rfloor}
\newcommand{\flip}[1]{#1^\circ}
\newcommand{\rows}[1]{\mathrm{rows}(#1)}
\newcommand{\cols}[1]{\mathrm{cols}(#1)}
\newcommand{\domleq}{\preceq}   
\newcommand{\domle}{\prec}
\newcommand{\supp}[1]{\mathrm{supp}(#1)}

\newcommand{\rib}[1]{\langle{#1}\rangle}
\newcommand{\ribschur}[1]{r_{#1}}

\begin{document}

\begin{frontmatter}

\title{Maximal supports and Schur-positivity among connected skew shapes}

\author{Peter R. W. McNamara\fnref{peter}}
\ead{peter.mcnamara@bucknell.edu}
\ead[url]{http://www.facstaff.bucknell.edu/pm040}

\author{Stephanie van Willigenburg\fnref{steph,nserc}}
\ead{steph@math.ubc.ca}
\ead[url]{http://www.math.ubc.ca/~steph}

\fntext[nserc]{The second author was supported in part by the National Sciences and Engineering Research Council of Canada.}

\address[peter]{Department of Mathematics, Bucknell University, Lewisburg, PA 17837, USA}

\address[steph]{Department of Mathematics, University of British Columbia, Vancouver, BC V6T 1Z2, Canada}

\begin{abstract}
The Schur-positivity order on skew shapes is defined by $B \leq A$ if the difference $s_A - s_B$ is Schur-positive.  It is an open problem to determine those connected skew shapes that are maximal with respect to this ordering.  A strong necessary condition for the Schur-positivity of $s_A- s_B$ is that the support of $B$ is contained in that of $A$, where the support of $B$ is defined to be the set of partitions $\lambda$ for which $s_\lambda$ appears in the Schur expansion of $s_B$.  We show that to determine the maximal connected skew shapes in the Schur-positivity order and this support containment order, it suffices to consider a special class of ribbon shapes.  We explicitly determine the support for these ribbon shapes, thereby determining the maximal connected skew shapes in the support containment order. 
\end{abstract}

\begin{keyword}
Symmetric function \sep Schur-positive \sep support \sep skew shape \sep ribbon

\MSC[2010] primary 05E05; secondary 05E10 \sep 06A06 \sep 20C30

\end{keyword}

\end{frontmatter}

\section{Introduction}\label{sec:intro} 

Among the familiar bases for the ring of symmetric functions, the Schur functions are commonly considered to be the most important basis.  Besides their elegant combinatorial definition, the significance of Schur functions in algebraic combinatorics stems from their appearance in other areas of mathematics.  More specifically, Schur functions arise in the representation theory of the
symmetric group and of the general and special linear groups.  Via Littlewood--Richardson coefficients,  Schur functions are also intimately tied to Schubert classes, which arise in algebraic geometry when studying the cohomology ring of the Grassmannian.  Furthermore, Littlewood--Richardson coefficients answer questions about eigenvalues of Hermitian matrices.  For more information on these and other
connections see, for example, \cite{Ful97} and \cite{Ful00}.

The aforementioned Littlewood-Richardson coefficients are nonnegative integers, and in the ring of symmetric functions they arise in two contexts: as the structure constants in the expansion of the product of two Schur functions $s_\lambda s_\mu$ as a linear combination of Schur functions, and as the structure constants in the expansion of a skew Schur function $s_{\lambda/\mu}$ as a linear combination of Schur functions. Consequently, these expansions give rise to the notion of a \emph{Schur-positive} function, i.e., when expanded as a linear combination of Schur functions, all of the coefficients are nonnegative integers.  Schur-positive functions have a particular representation-theoretic significance: if a homogeneous symmetric function $f$ of degree $N$ is Schur-positive, then it arises as the Frobenius image of some representation of the symmetric group $S_N$.  Moreover, $f(x_1, \ldots, x_n)$ is the character of a polynomial representation of the general linear group $\mathit{GL}(n, \mathbb{C})$.  Noting that $s_\lambda s_\mu$ is just a special type of skew Schur function \cite[p.~339]{ec2}, we will restrict our attention to skew Schur functions $s_A$, where $A$ is a skew shape.  Roughly speaking, our goal is to determine those $s_A$ that are the ``most'' Schur-positive.

Working towards making this goal more precise, 
one might next ask when expressions of the form
$s_A - s_B$
are Schur-positive, where $B$ is a skew shape.
Such questions have been the subject of much recent work, such as
\cite{BBR06, FFLP05, Kir04, KWvW08, LLT97, LPP07, McN08, McVw09b, Oko97}.  It is well known 
that these questions are
currently intractable when stated in anything close to full generality.  
A weaker condition than $s_A-s_B$ being Schur-positive is that the support of $s_B$ is contained in the support of $s_A$, 
where the support of $s_A$ is defined to be the set of those $\lambda$ for which $s_\lambda$ appears with nonzero coefficient when we expand $s_A$ as a linear combination of Schur functions.  
Support containment for skew Schur functions is directly relevant to the results of \cite{DoPy07, FFLP05, McN08}; let us give the flavor of just one beautiful result about the support of skew Schur functions.  There exist Hermitian matrices $A$, $B$ and $C=A+B$, with eigenvalue sets 
$\mu$, $\nu$ and $\lambda$ respectively, if and only if $\nu$ is in the support of $s_{\lambda/\mu}$.  (See the survey \cite{Ful00} and the references therein.)

Putting these questions in the following general setting will help put our work in context.  
We could define a reflexive and transitive binary relation on skew Schur functions by
saying that $B$ is related to $A$ if $s_A - s_B$ is Schur-positive.  To make this 
relation a partial order, we need to consider those skew shapes that yield the same
skew Schur function to be equivalent; see the sequence \cite{BTvW06, RSvW07, McvW09a} as well as \cite{vWi05,Gut09} for a study
of these equivalences.  Having done this, let us say that $[B] \leqs [A]$ if 
$s_A - s_B$ is Schur-positive, where $[A]$ denotes the equivalence class of $A$.  Since $s_A$ is homogeneous of degree $N$, where $N$ is the number of boxes of $A$,  $[A]$ and 
$[B]$ will be incomparable unless $A$ and $B$ have the same number $N$ of boxes, and we
let $\bigP{N}$ denote the poset of all equivalence classes $[A]$ such that $A$ has $N$ boxes.  Restricting
to skew shapes with 4 boxes, we get the poset $\bigP{4}$ shown in Figure~\ref{fig:p4}.
\begin{figure}[htbp]
\begin{center}
\includegraphics[width=0.5\textwidth]{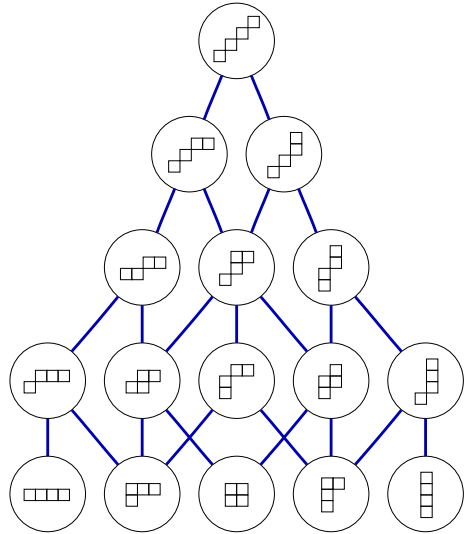}
\caption{$\bigP{4}$: All skew shapes with 4 boxes under the Schur-positivity order, up to equivalence. One can determine that $\bigP{N}$ is not graded when $N \geq 5$, and is not a join-semilattice when $N \geq 6$.}
\label{fig:p4}
\end{center}
\end{figure}

In a similar way, we can define a poset by $\eqsupp{B} \leqsupp \eqsupp{A}$ if the support of $B$ is contained in that of $A$, where $\eqsupp{A}$ denotes the support equivalence class of $A$.  We let $\Supp{N}$ denote the poset of all equivalence classes $\eqsupp{A}$ such that $A$ has $N$ boxes.  As a simple example, $\Supp{4}$ is identical to $\bigP{4}$.  When $N=5$, things become more interesting: let 
$A=(3,3,2,1)/(2,1,1)$, abbreviated as $A = 3321/211$, and $B = 3311/21$.  We see that 
\[
s_A = s_{32} + s_{311} + 2s_{221} + s_{2111}   \mbox{\ \ \ and\ \ \ } s_B = s_{32} + s_{311} + s_{221} + s_{2111}\  .
\]
So $[B] \les [A]$ in $\bigP{5}$, while $\eqsupp{B}=\eqsupp{A}$ in $\Supp{5}$.  
Our overarching goal when studying questions of Schur-positivity and support containment for skew Schur functions is to understand the posets $\bigP{N}$ and $\Supp{N}$.  

Despite their fundamental nature, it is easy to ask questions about $\bigP{N}$ and $\Supp{N}$ that sound simple but are not easy.  We will be interested in the maximal elements of these posets.  Well, in fact, it is easy to check that each of these posets has a unique maximal element, namely, the skew shape that consists of $N$ connected components, each of size 1, as is the case for $\bigP{4}$ in Figure~\ref{fig:p4}.  Instead, we will restrict our attention to connected skew shapes.  It will simplify our terminology if we make the following observation, which we will prove in Section~\ref{sec:reduction}: if two skew shapes $A$ and $B$ fall into the same equivalence class in $\bigP{N}$ or $\Supp{N}$, then $A$ and $B$ must have the same numbers of connected components, nonempty rows, and nonempty columns.  Therefore, without ambiguity, we can refer to the number of rows of an element of $\bigP{N}$ or $\Supp{N}$, or say if such an element is connected.  Our goal is to address the following two questions.

\begin{question}\label{que:schurpos} 
What are the maximal elements of the subposet of $\bigP{N}$ consisting of connected elements?
\end{question}

\begin{question}\label{que:support}
Similarly, 
what are the maximal elements of the subposet of $\Supp{N}$ consisting of connected elements?
\end{question} 

For example, we see in Figure~\ref{fig:p4} that $\bigP{4}=\Supp{4}$ has 7 connected elements, 4 of which are maximal among these connected elements.  See Figure~\ref{fig:p5connected} for the subposet of $\Supp{5}$ consisting of the connected elements, which equals the corresponding subposet for $\bigP{5}$.

Somewhat surprisingly, Question~\ref{que:schurpos} remains open.  
The following conjectural answer to Question~\ref{que:schurpos} is due to Pavlo Pylyavskyy and the first author.  It is well known \cite[Exer.~7.56(a)]{ec2} that a skew shape $A$ is equivalent in $\bigP{N}$ to its \emph{antipodal rotation} $\flip{A}$ (i.e., $\flip{A}$ is obtained from $A$ by rotating $A$ by 180$^\circ$).  

\begin{conjecture}\label{con:max} \
\renewcommand{\theenumi}{\alph{enumi}}
\begin{enumerate}
\item In the subposet of $\bigP{N}$ consisting of connected skew shapes, there are exactly  $N$ maximal elements.  More specifically, there is a unique maximal element with $l$ rows, for $l=1,\ldots,N$.  Each maximal element is an equivalence class consisting of a single skew shape $R$, along with its antipodal rotation $\flip{R}$ (if $R \neq \flip{R}$).  
\item Let $[R]$ denote the unique maximal element with $l$ rows.  To construct $R$ up to antipodal rotation, start with a grid $l$ boxes high and $N-l+1$ boxes wide and draw a line $L$ from the bottom left to the top right corner.  
Then $R$ consists of the boxes whose interior or whose top left corner point is intercepted by $L$.
\end{enumerate}
\end{conjecture}

\begin{example}
According to the conjecture, the unique maximal connected elements with 3 rows in $\bigP{7}$ and $\bigP{8}$, up to rotation, are shown in Figure~\ref{fig:billiard}.  These two examples are different in nature since the diagonal line in the second example goes through internal vertices of the grid, necessitating the ``top left corner point'' phrase in Conjecture~\ref{con:max}(b). 
Although none of our proofs are affected by the differing nature of these examples, there are implications for the discussion in Subsection~\ref{sub:christoffel}.

\begin{figure}[htbp]
\begin{center}
\includegraphics[width=0.55\textwidth]{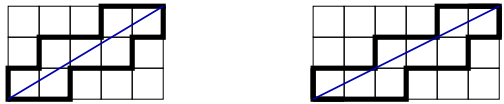}
\caption{The unique maximal connected element with 3 rows in $\bigP{7}$ (respectively $\bigP{8}$) is the equivalence class containing the skew shape outlined in bold on the left (resp.\ right), and its antipodal rotation.}
\label{fig:billiard}
\end{center}
\end{figure}
\end{example}

We use the letter $R$ because the resulting skew shape will always be a \emph{ribbon}, meaning that every pair of adjacent rows overlap in exactly one column.  Without using brute-force computation of skew Schur functions, we have verified Conjecture~\ref{con:max} for all $N \leq 33$: Remark~\ref{rem:nemeses} is a brief discussion of the ideas involved.  

In the present paper, we answer Question~\ref{que:support}.  As well as classifying the maximal connected elements of $\Supp{N}$, we can say exactly what the supports of these maximal elements are, showing that they take a particularly nice form.   The following theorem is our main result.

\begin{theorem}\label{thm:main}
\renewcommand{\theenumi}{\alph{enumi}}\
\begin{enumerate}
\item In the subposet of $\Supp{N}$ consisting of connected skew shapes, there are exactly $N$ maximal elements.  More specifically, there is a unique maximal element with $l$ rows, for $l=1,\ldots,N$.
\item For each such $l$, the corresponding maximal element is an equivalence class consisting of all those ribbons $R$ with $l$ rows  and with the following property: the lengths of any two nonempty rows of $R$ differ by at most one and the lengths of any two nonempty columns of $R$ differ by at most one.
\item For such $R$, a partition $\lambda$ is in the support of $s_R$ if and only if $\lambda$ has $N$ boxes and no more nonempty rows or columns than $R$.
\end{enumerate}
\end{theorem}

\begin{example}
The subposet of $\Supp{5}$ consisting of connected skew shapes, which equals that for $\bigP{5}$, is shown in Figure~\ref{fig:p5connected}, and is 
readily checked to be consistent with Theorem~\ref{thm:main}.  The support of a skew shape $A$ can be read off from the poset as the partitions that are less than  or equal to $A$.   
\begin{figure}[htbp]
\begin{center}
\includegraphics[width=0.8\textwidth]{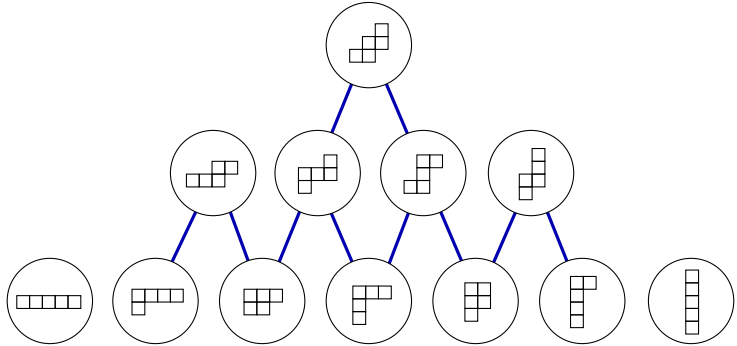}
\caption{The subposet of $\bigP{5}$ or $\Supp{5}$ consisting of the connected elements.  Each element includes the skew shape shown and, if different, its antipodal rotation.}
\label{fig:p5connected}
\end{center}
\end{figure}
\end{example}

\begin{example}\label{exa:fullsupp}
The maximal connected element of $\Supp{8}$ with three rows is $\eqsupp{R}$, where $R$ is shown on the right in Figure~\ref{fig:billiard}.  In contrast to the $\bigP{8}$ case, where $[R]$ contains just a single element and its antipodal rotation, in the support case $\eqsupp{R}$ contains three elements, namely the ribbons whose row lengths, read from top to bottom, are 233, 323 and 332.
By Theorem~\ref{thm:main}(c), the support of $s_R$ is 
\[ \{62, 611, 53, 521, 44, 431, 422, 332\}.\]  
\end{example}

The remainder of the paper is organized as follows.  In Section~\ref{sec:prelims}, we give the necessary symmetric function background and give precise definitions of many of the terms from this introduction.  In Section~\ref{sec:reduction}, we prove some foundational results that apply to both $\Supp{N}$ and $\bigP{N}$.  In particular, we reduce the problem by showing that the maximal connected elements of both $\bigP{N}$ and $\Supp{N}$ must be among those described in Theorem~\ref{thm:main}(b).  In Section~\ref{sec:fullsupport}, we prove Theorem~\ref{thm:main}(c), from which we will show (a) and (b) follow.  We conclude in Section~\ref{sec:conclusion} with open problems.

\section{Preliminaries}\label{sec:prelims} 

We follow the terminology and notation of \cite{Mac95,ec2}.

\subsection{Partitions and skew shapes} 
A \emph{composition} $\lambda=(\lambda_1, \ldots, \lambda_l)$ of $N$ is a sequence of positive integers whose sum is $N$.  We say that $N$ is the size of $\lambda$, 
denoted $|\lambda|$, and we call $l$ the \emph{length} of
$\lambda$ and denote it by $\ell(\lambda)$.  If $\lambda$ is a weakly decreasing sequence of positive integers then we say that $\lambda$ is a \emph{partition} of $N$, denoted $\lambda \vdash N$.  In this case, we will mainly think of $\lambda$ in terms of its \emph{Young diagram}, which is a
left-justified array of boxes that has $\lambda_i$ boxes in the $i$th row from the top.  For
example, if $\lambda = (4,4,3)$, which we will abbreviate as $\lambda = 443$, 
then the Young diagram of $\lambda$ is 
\setlength{\unitlength}{4mm}
\[
\begin{picture}(4,3)(0,0)
\put(0,3){\line(1,0){4}}
\put(0,2){\line(1,0){4}}
\put(0,1){\line(1,0){4}}
\put(0,0){\line(1,0){3}}
\multiput(0,3)(1,0){5}{\line(0,-1){2}}
\multiput(0,1)(1,0){4}{\line(0,-1){1}}
\end{picture}\ .
\]
We will often abuse terminology by referring to the Young diagram of $\lambda$ simply as $\lambda$.  For example, we will say that a partition $\mu$ is \emph{contained} in a partition $\lambda$ if the Young diagram of $\mu$ is contained in the Young diagram of $\lambda$.  
In this case, we define the \emph{skew shape} $\lambda/\mu$ to be the set of boxes in $\lambda$ that remain after we remove those boxes corresponding
to $\mu$.  For example, the skew shape $A = 443/2$ is represented as
\[
\begin{picture}(4,3)(0,0)
\put(2,3){\line(1,0){2}}
\put(0,2){\line(1,0){4}}
\put(0,1){\line(1,0){4}}
\put(0,0){\line(1,0){3}}
\multiput(2,3)(1,0){3}{\line(0,-1){1}}
\multiput(0,2)(1,0){5}{\line(0,-1){1}}
\multiput(0,1)(1,0){4}{\line(0,-1){1}}
\end{picture}\ .
\]
We will label skew shapes by simply
using single uppercase roman letters, as in the example above.  We 
write $|A|$ for the \emph{size} of $A$, which is simply the number of boxes in the skew shape $A$.
If $A = \lambda/\mu$ and $\mu$ is empty, then $A$ is said to be a \emph{straight shape}.

Certain classes of skew shapes will be of particular interest to us.  A skew shape $A$ is said to be \emph{disconnected} if it can be partitioned into two skew shapes $B$ and $C$ so that no box of $B$ shares a row or column with any box of $C$.  Otherwise $A$ is said to be \emph{connected}.
Playing a key role for us will be the class of \emph{ribbons}, which are connected skew shapes which don't contain the subdiagram  $22=
\setlength{\unitlength}{2mm}
\begin{picture}(2,2)(0,0.4)
\multiput(0,0)(0,1){3}{\line(1,0){2}}
\multiput(0,0)(1,0){3}{\line(0,1){2}}
\end{picture}$\,.
The skew shape above is not a ribbon whereas $R = \lambda/\mu = 433/22$, represented as
\setlength{\unitlength}{4mm}
\[
\begin{picture}(4,3)(0,0)
\put(2,3){\line(1,0){2}}
\put(2,2){\line(1,0){2}}
\put(0,1){\line(1,0){3}}
\put(0,0){\line(1,0){3}}
\multiput(2,3)(1,0){3}{\line(0,-1){1}}
\multiput(2,2)(1,0){2}{\line(0,-1){1}}
\multiput(0,1)(1,0){4}{\line(0,-1){1}}
\end{picture}\ ,
\]
certainly is.  Note that adjacent rows of a ribbon overlap in exactly one column, so we can completely classify a ribbon by the composition consisting of its row lengths from top to bottom.  We will write the ribbon above as $\rib{2,1,3}$, abbreviated as $\rib{213}$.

We will make significant use of the transpose operation on partitions.
For any
partition $\lambda$, we define its \emph{transpose} or \emph{conjugate} $\lambda^t$ to be the partition obtained by reading
the column lengths of $\lambda$ from left to right.  For example, $443^t = 3332$.
The transpose operation can be extended to skew shapes $A=\lambda/\mu$ by setting
$A^t = \lambda^t/\mu^t$.  
Another operation on the skew shape $A$ sends $A$ to its antipodal rotation, denoted $\flip{A}$, which is just $A$ rotated 180 degrees in the plane.  For example, $\flip{\rib{213}}=\rib{312}$ and, in general, the antipodal rotation of any ribbon will clearly just reverse the order of the row lengths.  

Given a skew shape $A$, a partition of particular interest will be $\rows{A}$ (resp.\ $\cols{A}$), defined to be the multiset of positive row (resp.\ column) lengths of $A$ sorted into weakly decreasing order.  For example, for the ribbon above we have $\rows{\rib{213}} =321$ and $\cols{\rib{213}}=3111$.  

We will compare partitions of equal size according to the \emph{dominance order}:
we will write $(\lambda_1, \lambda_2, \ldots, \lambda_r) \domleq  (\mu_1, \mu_2, \ldots, \mu_s)$ if 
\[
\lambda_1 + \lambda_2 + \cdots + \lambda_k \leq \mu_1 + \mu_2 + \cdots + \mu_k
\]
for all $k=1,2,\ldots,r$, where we set $\mu_i=0$ if $i > s$.
If is a nice exercise to show that if $\lambda \domleq \mu$ then $\mu^t \domleq \lambda^t$.

\subsection{Skew Schur functions and the Littlewood--Richardson rule}
While skew shapes are our main diagrammatical objects of study, our main algebraic
objects of interest are skew Schur functions, which we now define.  For a skew shape $A$, a
\emph{semi-standard Young tableau} (SSYT) of shape $A$ is a filling of the boxes of
$A$ with positive integers such that the entries weakly increase along the rows and 
strictly increase down the columns.  For example, 
\setlength{\unitlength}{4mm}
\[
\begin{picture}(4,3)(0,0)
\put(2,3){\line(1,0){2}}
\put(0,2){\line(1,0){4}}
\put(0,1){\line(1,0){4}}
\put(0,0){\line(1,0){3}}
\multiput(2,3)(1,0){3}{\line(0,-1){1}}
\multiput(0,2)(1,0){5}{\line(0,-1){1}}
\multiput(0,1)(1,0){4}{\line(0,-1){1}}
\put(2.3,2.2){1}
\put(3.3,2.2){2}
\put(0.3,1.2){1}
\put(1.3,1.2){1}
\put(2.3,1.2){2}
\put(3.3,1.2){3}
\put(0.3,0.2){5}
\put(1.3,0.2){7}
\put(2.3,0.2){7}
\end{picture}
\]
is an SSYT of shape $443/2$.  The \emph{content} of a filling $T$ is $c(T) = (c_1(T), c_2(T), \ldots)$, where $c_i(T)$ is the number of $i$'s in the filling.  
The \emph{skew Schur function} $s_A$ in the variables $(x_1, x_2, \ldots)$ is then defined by
\[
s_A = \sum_{T} x^{c(T)}
\]
where the sum is over all SSYTx $T$ of shape $A$, and  
\[
x^{c(T)} = x_1^{c_1(T)}
x_2^{c_2(T)} \cdots .
\]
For example, the SSYT above contributes the monomial $x_1^3 x_2^2 x_3 x_5 x_7^2$ to 
$s_{443/2}$.  We will also write $\ribschur{\alpha}$ to denote the skew Schur function of the ribbon with row lengths $\alpha$ from top to bottom.  

Although not obvious from the definition, it is well known  that $s_A = s_{\flip{A}}$ \cite[Exer.~7.56(a)]{ec2}.  For example, $\ribschur{213}=\ribschur{312}$.  While the identity map sends $s_A$ to $s_{\flip{A}}$, we denote by $\omega$ the well-known algebra endomorphism on symmetric functions defined by 
\begin{equation}\label{equ:omega}
\omega(s_\lambda) = s_{\lambda^t}
\end{equation}
for any partition $\lambda$.  Note that $\omega$ is an involution and it can be shown (see \cite[\S~I.5]{Mac95}, \cite[Thm.~7.15.6]{ec2} or, for the original proofs, \cite{Ait28,Ait31}) that $\omega$ extends naturally to skew Schur functions: $\omega(s_A) = s_{A^t}$.  

If $A$ is a straight shape, then $s_A$ is called simply a \emph{Schur function},
and some of the significance of Schur functions stems from the fact that they form a basis for the
symmetric functions.  Therefore, every skew Schur function can be written as a linear
combination of Schur functions.  A simple description of the coefficients in 
this linear combination is given by the celebrated \emph{Littlewood--Richardson rule}, which we 
now describe.  The \emph{reverse reading word} of an SSYT $T$ is the word obtained by 
reading the entries of $T$ from right to left along the rows, taking the 
rows from top to bottom.  For example, the SSYT above has reverse reading word 213211775.
An SSYT $T$ is said to be a \emph{Littlewood--Richardson filling} or  \emph{LR-filling} if, as we read the reverse reading word of $T$,
the number of appearances of $i$ always stays ahead of the number of appearances of
$i+1$, for $i=1,2,\ldots$.
The reader is invited to check that the only possible LR-fillings of $443/2$ have
reading words 112211322 and 112211332.  The Littlewood--Richardson rule \cite{LiRi34, Sch77, ThoThesis, Tho78} then states that
\[
s_{\lambda/\mu} = \sum_{\nu} c^\lambda_{\mu\nu} s_\nu\ ,
\]
where $c^\lambda_{\mu\nu}$ is the ubiquitous \emph{Littlewood--Richardson} coefficient, 
defined to be the number of LR-fillings of $\lambda/\mu$ with content $\nu$. 
For example, if $A = 443/2$, then $s_A = s_{441} + s_{432}$.
It follows that any skew Schur function can be written 
as a linear combination of Schur functions with all positive coefficients, and we thus say
that skew Schur functions are \emph{Schur-positive}.  

When $\lambda/\mu=\alpha$ is a ribbon, the expansion of $\ribschur{\alpha}$ in terms of Schur functions can be written in an alternative form, and it is this alternative form that will be most useful to us.  
By the size of an SSYT $T$, we will just mean the size of $A$, where $T$ has shape $A$.   If an SSYT $T$ of size $N$ has entries $\{1, 2, \ldots, N\}$, each necessarily appearing exactly once, then $T$ is said to be a \emph{standard Young tableau} (SYT).  The \emph{descent set} of an SYT $T$ is defined to be those entries $i$ for which $i+1$ appears in a lower row in $T$ than $i$.  
For example, the SYT
\setlength{\unitlength}{4mm}
\[
\begin{picture}(4,3)(0,0)
\put(0,3){\line(1,0){4}}
\put(0,2){\line(1,0){4}}
\put(0,1){\line(1,0){4}}
\put(0,0){\line(1,0){1}}
\multiput(0,3)(1,0){2}{\line(0,-1){3}}
\multiput(2,3)(1,0){3}{\line(0,-1){2}}
\put(0.3,2.2){1}
\put(1.3,2.2){2}
\put(2.3,2.2){3}
\put(3.3,2.2){6}
\put(0.3,1.2){4}
\put(1.3,1.2){5}
\put(2.3,1.2){7}
\put(3.3,1.2){9}
\put(0.3,0.2){8}
\end{picture}
\]
has size 9 and descent set $\{3,6,7\}$. 
Note that every composition $\alpha=(\alpha_1, \ldots, \alpha_l)$ of $N$ also naturally gives rise to a subset of $\{1,\ldots,N-1\}$, namely, $\{\alpha_1,  \alpha_1+\alpha_2, \ldots,
\alpha_1+\alpha_2+\cdots+\alpha_{l-1}\}$, which we denote by $S(\alpha)$.  
The following result is due to Ira Gessel \cite{Ges84}.
\begin{theorem}[\cite{Ges84}]\label{thm:ribexpansion}
For any composition $\alpha$ of $N$,
\[ 
\ribschur{\alpha} = \sum_{\lambda \vdash N} d_{\lambda \alpha} s_\lambda,
\]
where $d_{\lambda \alpha}$ equals the number of SYT of shape $\lambda$ and descent set $S(\alpha)$.   
\end{theorem}

As mentioned in the introduction, our main result concerns the 
\emph{support} of skew Schur functions.  The support $\supp{A}$ of $s_A$ 
is defined to be the set of those partitions $\nu$ for
which $s_\nu$ appears with nonzero coefficient when we expand $s_A$ as a linear combination of
Schur functions.  For example, we have $\supp{443/2} = \{441, 432\}$.  
We sometimes talk of the support of $A$, by which we mean the support of $s_A$.  

\section{Reducing the problem}\label{sec:reduction}

 At face value, Questions~\ref{que:schurpos} and \ref{que:support} require us to consider all connected skew shapes.  In this section, we will show that it suffices to consider only ribbons, and then show that the maximal connected elements must be ribbons whose multisets of row lengths and column lengths take a certain form.  As promised, we will also prove our earlier assertions about necessary conditions for two skew shapes to be equivalent.  Except where specified, the deductions about maximal connected elements in this section apply to both $\bigP{N}$ and $\Supp{N}$.  
 
We first need a preliminary result about the elements of the support of a skew shape. 
It 
appears in our notation in \cite{McN08}, although earlier proofs can be found in \cite{Lam77, Zab}.

\begin{lemma}\label{lem:extreme_fillings}  Let $A$ and $B$ be skew shapes.
 \renewcommand{\theenumi}{\alph{enumi}}
 \begin{enumerate}
\item If $\lambda \in \supp{A}$, then 
\[
\rows{A} \domleq \lambda \domleq \cols{A}^t, 
\]
 and both $s_{\rows{A}}$ and $s_{\cols{A}^t}$ appear with coefficient 1 in the Schur expansion of $s_A$. 
\item Consequently, if $\supp{A} \supseteq \supp{B}$, then
\[
\rows{A} \domleq \rows{B}  \mbox{\ \ and \ \ } \cols{A} \domleq \cols{B}.
\]
\end{enumerate}
\end{lemma}

The following assertions were partially stated in the introduction, and allow us to talk about the number of rows, columns and connected components of elements of $\bigP{N}$ and $\Supp{N}$  without ambiguity.  
 
 \begin{lemma}\label{lem:necconds}  Suppose skew shapes $A$ and $B$ are in the same equivalence class in $\Supp{N}$, i.e., $\eqsupp{A}=\eqsupp{B}$.   (In particular, this is the case if $[A]=[B]$ in $\bigP{N}$.)  Then the following conditions are true:
 \renewcommand{\theenumi}{\alph{enumi}}
 \begin{enumerate}
 \item \label{ite:rows_equal} $\rows{A}=\rows{B}$ and $\cols{A}=\cols{B}$.  In particular, $A$ and $B$ have the same number of nonempty rows, and similarly for columns; 
 \item $A$ and $B$ have the same number of connected components; 
 \item \label{ite:ribbon} if $A$ is a ribbon, then so is $B$.
 \end{enumerate}
 \end{lemma}

\begin{proof}
If $\eqsupp{A}=\eqsupp{B}$ then, by Lemma~\ref{lem:extreme_fillings}(b), we have $\rows{A} \domleq \rows{B}$ and $\rows{B} \domleq \rows{A}$.  Thus $\rows{A}=\rows{B}$ and hence $A$ and $B$ have the same number of nonempty rows.  Similarly, $\cols{A}=\cols{B}$ and $A$ and $B$ have the same number of nonempty columns.  

Parts (b) and (c) follow immediately from \cite[Cor\ 4.1]{McN08}.  Indeed, this corollary (with $k=2$) states that if $\supp{A}=\supp{B}$ then $A$ and $B$ must have the same ``row overlap partitions,'' meaning that if $A$ has $r_i$ pairs of adjacent nonempty rows that overlap in exactly $i$ columns, then so must $B$.  The number of connected components of $A$ is $r_0+1$.  We also see that $A$ is a ribbon if and only if $r_i=0$ for $i\neq1$.
\end{proof}

We now make our first major reduction in the number of connected skew shapes we must consider when tackling either Question~\ref{que:schurpos} or Question~\ref{que:support}.    

\begin{proposition}\label{pro:ribbons_only}
Suppose $A$ is a connected skew shape of size $N$  that is not a ribbon.  Then there exists a ribbon $R$ such that $[A] \les [R]$ and $\eqsupp{A} \lesupp \eqsupp{R}$.  
\end{proposition}

\begin{proof}
Since $A$ is connected but not a ribbon, there must exist an $i$ such that the $i$th and $(i+1)$st (counting from the top) rows of $A$ overlap in at least two columns.  Form a new skew shape $A'$ by sliding rows $i$ and higher one position to the right.  There is a natural content-preserving injection from the set of LR-fillings of $A$ into the set of LR-fillings of $A'$.  Indeed, if $T$ is an LR-filling of $A$, then sliding the entries of $T$ in rows $i$ and higher one position to the right gives a filling $T'$ of $A'$.  Since this slide preserves content and the SSYT and LR properties, we have the desired injection.

Repeating this procedure as necessary gives a content-preserving injection from the set of LR-fillings of $A$ into the set of LR-fillings of a ribbon $R$, and hence $[A] \leqs [R]$ and $\eqsupp{A} \leqsupp \eqsupp{R}$.   It remains to show strict inequality by showing that the injection is not a bijection.  While $\rows{R} = \rows{A}$ by construction, we see that $\ell(\cols{A}) < \ell(\cols{A'})$ and hence $\ell(\cols{A}) < \ell(\cols{R})$.  Applying Lemma~\ref{lem:necconds}(\ref{ite:rows_equal}) then gives that $[A] \neq [R]$ and $\eqsupp{A} \neq \eqsupp{R}$, and so $[A] \les [R]$ and $\eqsupp{A} \lesupp \eqsupp{R}$, as required.
\end{proof}

In our search for maximal connected elements, Proposition~\ref{pro:ribbons_only} combined with Lemma~\ref{lem:necconds}(\ref{ite:ribbon}) allows us to restrict our attention to ribbons.
We now show that results from \cite{KWvW08} allow us to do even better.  

\begin{definition}
We say that a ribbon is \emph{row equitable} (resp.\ \emph{column equitable}) if all its row (resp.\ column) lengths differ by at most one.  A ribbon is said to be \emph{equitable} if it both row and column equitable.
\end{definition}

\begin{proposition}\label{pro:equitable_only} \
\begin{enumerate}
\renewcommand{\theenumi}{\alph{enumi}}
\item Suppose $A$ is a ribbon which is not row equitable.   Then there exists a row equitable ribbon $R$ such that $[A] \les [R]$ and $\eqsupp{A} \lesupp \eqsupp{R}$.
\item Suppose $A$ is a ribbon which is not column equitable.   Then there exists a column equitable ribbon $R$ such that $[A] \les [R]$ and $\eqsupp{A} \lesupp \eqsupp{R}$.
\item Suppose $A$ is a ribbon which is not equitable.   Then there exists an equitable ribbon $R$ such that $[A] \les [R]$ and $\eqsupp{A} \lesupp \eqsupp{R}$.
\end{enumerate}
\end{proposition}

\begin{proof} We first prove (a).  Suppose the length of row $i$ of $A$ is at least two bigger than the length of row $j$ of $A$, where we choose $i$ and $j$ so that $|i-j|$ is minimal.  If $i=j-1$, then let $R'$ be the ribbon obtained from $A$ by making row $i$ one box shorter and row $j$ one box longer.  Applying \cite[Cor.\ 2.8]{KWvW08}, we get immediately that $s_{R'} - s_A$ is Schur-positive.  If $i=j+1$ then, since $s_A=s_{\flip{A}}$, we can apply exactly the same technique with $\flip{A}$ in place of $A$.  

Now suppose that $|i-j|>1$.  We see that, since $|i-j|$ was chosen to be minimal, there must exist an adjacent sequence of rows of $A$ or $\flip{A}$ of lengths $a+1, a, \ldots, a, a-1$ read from top to bottom.   Let $R'$ be the ribbon obtained by giving all rows in the sequence the length $a$.  This is exactly the situation necessary for \cite[Thm.\ 2.13]{KWvW08}: we deduce that $s_{R'} - s_A$ is Schur-positive.   

 We conclude that $[A] \leqs [R']$, so $\eqsupp{A} \leqsupp \eqsupp{R'}$.  To obtain strict inequalities, apply Lemma~\ref{lem:necconds}(\ref{ite:rows_equal}), observing that in all cases, the resulting ribbon $R'$ satisfies $\rows{R'} \domle \rows{A}$.
 
Since $\rows{R'} \domle \rows{A}$, if we
repeat this whole procedure, now working with $R'$ in place of $A$, we will eventually arrive at
a row equitable ribbon $R$ such that $[A] \les [R]$ and $\eqsupp{A} \lesupp \eqsupp{R}$, as required.

To prove (b), we invoke the $\omega$ involution from~$\eqref{equ:omega}$.  By definition,   
$\omega$ preserves the properties of Schur-positivity and support containment.  More specifically, for skew shapes $A$ and $B$, $[A] \leqs [B]$ if and only if $[A^t] \leqs [B^t]$, and similarly for $\leqsupp$.  Therefore, we can apply the procedure from (a) to $A^t$, which is not row equitable, to yield a row equitable ribbon $R$ such that $[A^t] \les [R]$ and $\eqsupp{A^t} \lesupp \eqsupp{R}$.  Therefore, we have a column equitable ribbon $R^t$ such that 
$[A] \les [R^t]$ and $\eqsupp{A} \lesupp \eqsupp{R^t}$.

To prove (c), we first apply (b) to produce a column equitable ribbon $A'$ such that 
\begin{equation}\label{equ:colequitable}
[A] \leqs [A'] \mbox{\ \ and\ \ }  \eqsupp{A} \leqsupp \eqsupp{A'};
\end{equation} 
we now have weak inequality since $A$ may itself  be column equitable.  We now apply the operations of (a) to $A'$ to produce a row equitable ribbon $R$ such that 
\begin{equation}\label{equ:rowequitable}
[A'] \leqs [R] \mbox{\ \ and\ \ }  \eqsupp{A'} \leqsupp \eqsupp{R}.
\end{equation} 
The key idea is that in applying the operations of (a), we can check that $\cols{R} \domleq \cols{A'}$.  This means that we must have $\cols{R} = \cols{A'}$ since $A'$ is column equitable, and so $R$ is column equitable. Since $A$ is not equitable, the inequalities in at least one of~\eqref{equ:colequitable} and~\eqref{equ:rowequitable} must be strict.  Therefore, $R$ is an equitable ribbon with the required properties.  
\end{proof}
 
By Lemma~\ref{lem:necconds}, if a skew shape $B$ satisfies $B \in [R]$ or $B \in \eqsupp{R}$, and $R$ is an equitable ribbon, then $B$ is a ribbon with the same multisets of row lengths and column lengths as $R$, and so must also be an equitable ribbon.  Combining this fact with Proposition~\ref{pro:ribbons_only} and Proposition~\ref{pro:equitable_only}(c) allows us to conclude the following reduction.  

\begin{corollary}\label{cor:full_reduction}
If $[A]$ (resp.\ $\eqsupp{A}$) is a maximal connected element of $\bigP{N}$ (resp.\ $\Supp{N}$) then $A$ 
is an equitable ribbon.
\end{corollary}

The equitable ribbons are exactly those of interest in answering Question~\ref{que:support}, which we will do in detail in the next section.  The following example of Corollary~\ref{cor:full_reduction} will be useful in the next section; we include it here because it is also relevant to the $\bigP{N}$ case.

\begin{example}\label{exa:lengthone}
Suppose $[A]$ (resp.\ $\eqsupp{A}$) is a maximal connected element of $\bigP{N}$ (resp.\  $\Supp{N}$) such that $A$ has at least one row of length 1 and at least one column of length 1.  
 Applying Corollary~\ref{cor:full_reduction}, we have that $A$ is a ribbon all of whose rows and columns are of length 1 or 2.   Since $A$ has no columns of length $3$, the only rows of $A$ that can have length 1 are the top and bottom rows.  Similarly, only the first and last columns of $A$ can possibly have length 1.   
 
This information is enough to tell us that there are just two possibilities for $A$.  If both the first and last rows (resp.~columns) of $A$ have length 1, then $A$ will have no columns (resp.~rows) of length 1.  Thus the first row and last column, or the last row and first column, must be length 1, implying that $N$ is odd.  
The two possibilities are $A$ and $\flip{A}$, and so constitute the same element $\{A, \flip{A}\}$ of $\bigP{N}$ or $\Supp{N}$.  Note that there are no other elements equivalent to $A$ or $\flip{A}$ since, by Lemma~\ref{lem:necconds}(\ref{ite:rows_equal}), such an element would also have rows and columns of length 1 and would also be equitable, implying its equality to $A$ or $\flip{A}$.  When $N=5$, this element $\{A, \flip{A}\}$ appears at the top of Figure~\ref{fig:p5connected}.
\end{example}

Let us make one final note that applies to both $\bigP{N}$ and $\Supp{N}$.  Conjecture~\ref{con:max} and Theorem~\ref{thm:main} both include the assertion that the relevant poset contains exactly $N$ maximal connected elements.  It is not hard for us now to show that each poset contains \emph{at least} $N$ maximal connected elements.  

\begin{lemma}\label{lem:incomparable}
Let $R$ and $R'$ be ribbons of size $N$.  Then $\eqsupp{R}$ and $\eqsupp{R'}$ are incomparable in $\Supp{N}$ if $R$ and $R'$ have a different number of rows.  Consequently, $[R]$ and $[R']$ are incomparable in $\bigP{N}$ if $R$ and $R'$ have a different number of rows.  
\end{lemma}

\begin{proof}
Suppose, without loss of generality, that $\eqsupp{R} \leqsupp \eqsupp{R'}$.  Applying Lemma~\ref{lem:extreme_fillings}(b), we have that $\rows{R'} \domleq \rows{R}$ and hence $\ell(\rows{R}) \leq \ell(\rows{R'})$.  Similarly, $\ell(\cols{R}) \leq \ell(\cols{R'})$.  However, for any ribbon $A$ of size $N$, we can check that  $\ell(\rows{A}) + \ell(\cols{A}) = N + 1$.  We conclude that $\ell(\rows{R}) = \ell(\rows{R'})$, as required.

The second assertion of the lemma follows because incomparability in $\Supp{N}$ implies incomparability in $\bigP{N}$.  
\end{proof}

Lemma~\ref{lem:incomparable} has the following immediate consequence.

\begin{corollary}\label{cor:onemaximal}
For each $l =  1, \ldots, N$,  $\bigP{N}$ and $\Supp{N}$ each must contain at least one maximal connected element with $l$ rows.  
\end{corollary}

While we now have all the foundation we need to work on proving Theorem~\ref{thm:main} in the next section, we can say a bit more specifically about $\bigP{N}$ that is relevant to Conjecture~\ref{con:max}.

\begin{remark}\label{rem:nemeses}
Corollary~\ref{cor:full_reduction} 
gives significant information about the maximal connected elements in $\bigP{N}$.  However it doesn't, for example, explain why $[\rib{323}] \leqs [\rib{233}]$.  Although Question~\ref{que:schurpos} remains open, we offer some methods to further reduce the possibilities for the maximal elements.  

Suppose a ribbon $A$ has all rows of lengths $a$ and $a+1$.  If the top or bottom row of $A$ has length $a+1$ while the adjacent row has length $a$, then \cite[Cor.\ 2.10]{KWvW08} tells us that $[A] \les [R]$, where $R$ is obtained from $A$ by switching the lengths of the two rows in question.  This shows, for example, that $[\rib{323}] \leqs [\rib{233}]$.  

The next simplest inequality unexplained by the methods described so far is $[\rib{23222}] \leqs [\rib{22322}]$.  Using a well-known skew Schur function identity and \cite[Thm.~5]{LPP07}, we can explain this inequality and any inequality necessary for Conjecture~\ref{con:max} with $N \leq 33$.  The general idea is like that in the previous paragraph where we switch adjacent row lengths $a$ and $a+1$ under the right conditions, except that now the switch doesn't have to take place at either end of the ribbon.   
Since our technique doesn't work in every case, we will only sketch one example.  Consider the difference $f = \ribschur{22322} - \ribschur{23222}$.  We have the identity $\ribschur{22}\ribschur{322} = \ribschur{22322} + \ribschur{2522}$ \cite[\S~169]{Mac04} and similarly $\ribschur{222}\ribschur{32} = \ribschur{23222} + \ribschur{2522}$.  Thus $f = \ribschur{22}\ribschur{322} - \ribschur{222}\ribschur{32}$.  Observe that we have written a difference of connected skew Schur functions as a difference of products of skew Schur functions, making the results of \cite{LPP07} applicable.  Applying the $\wedge$ and $\vee$ operations of \cite{LPP07} to $\ribschur{222}\ribschur{32}$ written as $s_{432/21} s_{53/21}$
yields $\ribschur{22}\ribschur{322}$.  By \cite[Thm.\ 5]{LPP07}, $f$ is Schur-positive, as desired.

The simplest inequality unexplained by all our methods involves $N=34$ with 14 rows: 
\[
[\rib{23232233223232}] \leqs [\rib{22323232323232}].
\]
\end{remark}

\section{Maximal support}\label{sec:fullsupport}

In this section, we prove Theorem~\ref{thm:main}.  Our first step is to reduce our task to proving Theorem~\ref{thm:fullsupport} below, which is the heart of the proof.  As mentioned at the end of the Introduction, we claim that proving Theorem~\ref{thm:main}(c) will suffice to prove parts (a) and (b).  Indeed, suppose (c) is true, i.e., a partition $\lambda$ is in the support of an equitable ribbon $R$ if and only if $|\lambda|=|R|$ and $\lambda$ has no more nonempty rows or columns than $R$.  When we consider maximal connected elements, we showed in Corollary~\ref{cor:full_reduction} that we can restrict our attention to ribbons.  By Lemma~\ref{lem:incomparable},  we can consider those ribbons with $l$ rows separately from those with any other number of rows.  Thus, let us fix $N$ and the number of rows $l$.  Note that, for ribbons, this also fixes the number of columns as $N-l+1$.  

Corollary~\ref{cor:onemaximal} gives that $\Supp{N}$ contains at least one maximal connected element with $l$ rows.  Moreover, by Corollary~\ref{cor:full_reduction}, we know that if $\eqsupp{R}$ is a maximal connected element of $\Supp{N}$, then $R$ is an equitable ribbon.   By Theorem~\ref{thm:main}(c), all equitable ribbons $R$ with $l$ rows  have the same support.  Therefore, all equitable ribbons with $l$ rows must constitute the same maximal element $\eqsupp{R}$ of $\Supp{N}$.  Thus (b) holds.  This also implies that $\Supp{N}$ contains exactly one maximal connected element with $l$ rows, namely $\eqsupp{R}$, and so (a) holds.

It remains to prove Theorem~\ref{thm:main}(c).  If fact, the ``only if'' direction of Theorem~\ref{thm:main}(c) is easy to check.  Let $R$ be an equitable ribbon of size $N$.  If  $\lambda \in \supp{R}$ then we know $|\lambda|=N$.   By Lemma~\ref{lem:extreme_fillings}(a), we have that $\rows{R} \domleq \lambda$, which implies that $\ell(\lambda) \leq \ell(\rows{R})$.  
We also have $\lambda_1 \leq \ell(\cols{R})$, since the inequality  $\lambda \domleq \cols{R}^t$ from Lemma~\ref{lem:extreme_fillings}(a) gives $\lambda_1 \leq (\cols{R}^t)_1 = \ell(\cols{R})$.

Therefore, to prove Theorem~\ref{thm:main}, it remains to show the following result.

\begin{theorem}\label{thm:fullsupport}
Let $R$ be an equitable ribbon with $|R|=N$.  If $\lambda$ with $|\lambda|=N$
has no more rows or columns than $R$, then $\lambda \in \supp{R}$.  
\end{theorem}

Not only will this suffice to prove Theorem~\ref{thm:main}, it also shows the nice result that the support of an equitable ribbon $R$ consists of all those partitions of the appropriate size that fit inside the same size rectangle as $R$.  In other words, the support of an equitable ribbon is as large as it can possibly be.

\begin{proof}[Proof of Theorem~\ref{thm:fullsupport}]
Let $\alpha$ denote the composition of row lengths of $R$ from top to bottom, and abbreviate $\ell(\alpha)$ by $l$.  By Theorem~\ref{thm:ribexpansion}, we wish to show that there exists an SYT $T$ of shape $\lambda$ and descent set $S(\alpha)$, for all $\lambda$ satisfying the hypotheses of the theorem.  We will proceed by induction on $l$, noting that the theorem is trivially true when $l=1$.  Assuming $l\geq 2$, our induction hypothesis is that if $\mu \vdash N-\alpha_l$ has at most $l-1$ rows and at most $N-\alpha_l - (l-1)+1$ columns, then there exists an SYT $T'$ of shape $\mu$ with descent set $S(\alpha_1, \ldots, \alpha_{l-1})$.  We will first show some restrictions that can be placed on $R$ and $\alpha$.  Then we will show that a horizontal strip $\lambda/\mu$ can be removed from $\lambda$ in such a way that $\mu$ satisfies the dimension requirements of the induction hypothesis.  We will fill the subpartition $\mu$ of $\lambda$ with $T'$ and the horizontal strip $\lambda/\mu$ with the numbers $N-\alpha_l+1, \ldots, N$.  While the resulting SYT may not have descent set $S(\alpha)$, we will show that a permutation of the entries will yield an SYT of shape $\lambda$ and descent set $S(\alpha)$.  

\subsection*{Restrictions on $R$ and $\alpha$} Since $R$ is equitable, there exists a positive integer $a$ (resp.~$b$) such that all rows (resp.~columns) of $R$ have length $a$ or $a+1$ (resp.\ $b$ or $b+1$), with at least one row having length $a$ (resp.\ $b$).  It will be helpful to restrict to the case where $a \geq b$.  
If this is not the case, then we can apply the $\omega$ involution of \eqref{equ:omega} and work with $R^t$ instead: $\lambda$ has no more rows and columns than $R$ if and only if $\lambda^t$ has no more rows and columns than $R^t$, and $\lambda \in \supp{R}$ if and only if $\lambda^t \in \supp{R^t}$.

Before proceeding, it is worth taking note of some restrictions that can be placed on $\alpha$ when $a=1$.  Since $b \leq a$, we are in the situation of Example~\ref{exa:lengthone}, implying that all rows of $R$ have length 2, except that exactly one of the top and bottom rows of $R$ has length 1.  If $\alpha_l=1$ while $\alpha_1=2$, then we can work with $\flip{R}$ instead of $R$, exploiting the facts that $\supp{\flip{R}} = \supp{R}$ and that $\flip{R}$ has the same number of rows and columns as $R$.  We conclude that we can restrict to the case when the only row of $R$ that can possibly have length 1 is the top row.  In other words, 
\begin{equation}\label{equ:rowlengths}
\alpha_i>1 \mbox{ for } 2\leq i\leq l.
\end{equation}

\subsection*{Constructing the horizontal strip $\lambda/\mu$} It will be helpful to denote the elements of $S(\alpha)$ by $N_1, N_2, \ldots, N_{l-1}$ in increasing order.  
Starting with an empty Young diagram of shape $\lambda$, we wish to appropriately insert the entries $N_{l-1}+1, N_{l-1}+2, \ldots, N$ into $\lambda$ so that the remaining empty boxes form a Young  diagram of shape $\mu$, where $\mu$ satisfies the dimension conditions of the induction hypothesis.  So that none of $N_{l-1}+1, N_{l-1}+2, \ldots, N$ will be a descent, $\lambda/\mu$ must be a horizontal strip.  Thus let us first check that $\lambda$ has at least $N-N_{l-1}=\alpha_l$ columns.  The first case is that $\alpha_l=a$.  If $\lambda_1\leq \alpha_l-1$, then $|R| = |\lambda| \leq l(a-1)$, since $\lambda$ is a partition with at most $l$ rows.  But then $R$ must have a row of length less than $a$, which is a contradiction.  On the other hand, if $\alpha_l=a+1$ and $\lambda_1\leq \alpha_l-1$, we deduce that $R$ has all rows of length $a$, again a contradiction as $\alpha _l = a+1$.  Thus $\lambda$ has at least $\alpha_l$ columns.  

We must next show that there is a way to choose this horizontal strip $\lambda/\mu$ so that $\mu$ satisfies the dimension requirements of the induction hypothesis.   See Figure~\ref{fig:mainproof} for a schematic representation of the situation, with $\lambda$ corresponding to the shaded region.
\begin{figure}[htbp]
\begin{center}
\begin{tikzpicture}[scale=0.9]
\filldraw[fill=blue!20!white] (0,0) -- (2,0) -- (2,0.4) -- (3,1.8) -- (4,2) -- (6,3.6) -- (6.6,3.6) -- (6.6,4) -- (0,4) -- (0,0);
\draw (0,0) rectangle (7,4);
\draw (0,0.4) rectangle (5,4);
\draw[densely dashed] (5,3.2) rectangle (5.4,3.6);
\draw (1,0.2) node {$r$};
\draw (0.5,2) node {$l-1$};
\draw (3 ,4.3) node {$N-l+1$};
\draw (3,3.65) node {$N_{l-1}-(l-1)+1$};
\draw (5.8,3.8) node {$c$};
\draw (5.2,3.4) node {$x$};
\draw (7.2,2) node {$l$};
\end{tikzpicture}
\end{center}
\caption{The dimensions in the proof of Theorem~\ref{thm:fullsupport}, with $\lambda$ shaded.}
\label{fig:mainproof}
\end{figure}
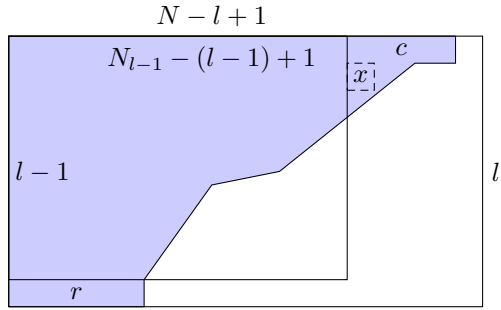
Since $\lambda/\mu$ must be a horizontal strip, we require that all $c$ columns to the right of column $N_{l-1}-(l-1)+1$ have at most one box.  It suffices to show that $\lambda$ does not have a box in row 2 and column $N_{l-1}-(l-1)+2$, marked with an $x$ in the figure.  So suppose $\lambda$ has a box in that location $x$.  Then $N = |\lambda| \geq 2(N_{l-1}-(l-1)+2) = 2(N-\alpha_l-l+3)$.  Thus 
\begin{equation}\label{equ:forcontradiction}
N-2\alpha_l-2l+6 \leq 0.
\end{equation}
We consider the cases $\alpha_l=a$ and $\alpha_l=a+1$ separately.  If $\alpha_l=a$, then $N \geq la$ and \eqref{equ:forcontradiction} implies $la-2a-2l+6 \leq 0$, which can be rewritten as 
\begin{equation}\label{equ:impossible}
(l-2)(a-2)+2 \leq 0. 
\end{equation}
Since we are assuming $l\geq2$, we know by \eqref{equ:rowlengths} that $a=\alpha_l \geq 2$.   Thus \eqref{equ:impossible} gives a contradiction.  If $\alpha_l=a+1$, then $N\geq la+1$ and \eqref{equ:forcontradiction} implies 
\[
(l-2)(a-2)+1 \leq 0.  
\]
If $a\geq2$, then we have our necessary contradiction.  If $a=1$, then by~\eqref{equ:rowlengths}, we must have $\alpha=(1,2,2,\ldots,2)$, and so $N=2l-1$.  Plugging this value of $N$ into \eqref{equ:forcontradiction} yields a contradiction with $\alpha_l=2$.  Therefore, both when $\alpha_l=a$ and $\alpha_l=a+1$, we have shown that there is no box of $\lambda$ at position $x$.  

We can now state our rule for constructing $\lambda/\mu$.  We need that $\mu$ has at most $l-1$ rows and at most $N_{l-1}-(l-1)+1$ columns.  Therefore, whenever they exist, $\lambda/\mu$ will include the $c$ boxes to the right of column $N_{l-1}-l+2$, and the $r$ boxes in row $l$.  We will also require that $\lambda/\mu$ include at least one box from the bottom row of $\lambda$ and, if $r =0$, that $\lambda/\mu$ include a box from the rightmost column.  The remaining entries of $\lambda/\mu$ can be chosen arbitrarily so that $\lambda/\mu$ is a horizontal strip.  Our next task is to show that all this can be done with just the $\alpha_l$ boxes that $\lambda/\mu$ is allowed to have.  
\begin{itemize}
\item If $r=c=0$ then, by \eqref{equ:rowlengths}, there are no difficulties: $\lambda/\mu$ will take one box from the bottom row of $\lambda$, a box (possibly the same) from the rightmost column, and choose the remaining boxes so that $\lambda/\mu$ will be a horizontal strip.
\item Suppose $r>0$ and $c=0$.  We require that $r \leq \alpha_l$.  If $r > \alpha_l$, then $|\lambda| \geq l(\alpha_l+1)$.  But we know that $N= |\lambda| \leq l(a+1)-1$, so we get a contradiction because $\alpha_l\geq a$.  
\item Suppose $r=0$ and $c>0$.  We have $c \leq N-l+1-(N_{l-1}-(l-1)+1)$ and, since $N_{l-1}=N-\alpha_l$, we have $c \leq \alpha_l-1$.  Therefore, $\lambda/\mu$ will be able to include these $c$ boxes and still be able to take at least one box from the bottom row of $\lambda$.  Because $c>0$, it will have automatically have included a box from the rightmost column.
\item Finally, suppose $r, c >0$.  We require that $r+c \leq \alpha_l$.  Since $r, c >0$, we have $N=|\lambda| \geq r(l-1)+(N_{l-1}-(l-1)+1)+c$.  Replacing $N_{l-1}$ by $N-\alpha_l$ we deduce that $l(r-1)-r+c-\alpha_l+2 \leq 0$.  Using $l \geq 2$ with $r\geq 1$, we get that $2(r-1)-r+c-\alpha_l+2 \leq 0$, which simplifies to $r+c \leq \alpha_l$, as required.  Since $r>0$, $\lambda/\mu$ will automatically include a box from the bottom row.
\end{itemize}
We conclude that a horizontal strip $\lambda/\mu$ of size $\alpha_l$ can be chosen so that $\mu$ satisfies the dimension requirements of the induction hypothesis, and so that $\lambda/\mu$ includes at least one box in the bottom row of $\lambda$ and, if $r = 0$, at least one box from the rightmost column of $\lambda$.

\subsection*{Filling $\lambda$ with an SYT} We can now designate an SYT $T$ of shape $\lambda$.  Considering $\mu$ and $\lambda/\mu$ as subsets of the boxes of $\lambda$, we fill the boxes of $\mu$ with $T'$ of the induction hypothesis, and those of $\lambda/\mu$ with $N_{l-1}+1, N_{l-1}+2, \ldots, N$ from left to right.  
By the induction hypothesis, the descent set of $T$ includes $N_1, N_2, \ldots, N_{l-2}$.  It may or may not include $N_{l-1}$ and, since $\lambda/\mu$ is a horizontal strip, the descent set of $T$ includes no numbers not of the form $N_j$ for some $j$.  If the entry $N_{l-1}$ is on a higher row than the entry $N_{l-1}+1$ in $T$, then $T$ has the required descent set $\{N_1, \ldots, N_{l-1}\}=S(\alpha)$ and we have proved the theorem.  Our final task is to assume that $N_{l-1}$ is not a descent in $T$, and show that we can then permute the entries of $T$ so that it continues to be an SYT but so that its descent set becomes $\{N_1, \ldots, N_{l-1}\}$.  

\begin{example}\label{permuteentries:exa}
To illustrate the construction of $T$ and the necessary permutations, consider $\alpha=333333$, so that $N=18$, $l=6$, $N_{l-1}=15$ and the entries of $\lambda/\mu$ will be 16, 17 and 18.    

First consider $\lambda=(10,4,4)$.  An SYT $T$ that would be consistent with our construction above would be
\setlength{\unitlength}{5mm}
\[
\begin{picture}(10,3)(0,0)
\put(0,3){\line(1,0){8}}
\put(0,2){\line(1,0){8}}
\put(0,1){\line(1,0){4}}
\put(0,0){\line(1,0){3}}
\multiput(0,3)(1,0){4}{\line(0,-1){3}}
\put(4,3){\line(0,-1){2}}
\multiput(5,3)(1,0){4}{\line(0,-1){1}}
\multiput(3.1,0)(0.2,0){5}{\line(1,0){0.1}}
\multiput(8.1,2)(0.2,0){10}{\line(1,0){0.1}}
\multiput(8.1,3)(0.2,0){10}{\line(1,0){0.1}}
\multiput(4,0.9)(0,-0.2){5}{\line(0,-1){0.1}}
\multiput(9,2.9)(0,-0.2){5}{\line(0,-1){0.1}}
\multiput(10,2.9)(0,-0.2){5}{\line(0,-1){0.1}}
\put(0.3,2.25){1}
\put(1.3,2.25){2}
\put(2.3,2.25){$\overline{3}$}
\put(3.3,2.25){$\overline{6}$}
\put(4.3,2.25){8}
\put(5.3,2.25){$\overline{9}$}
\put(6.15,2.25){11}
\put(7.15,2.25){$\overline{12}$}
\put(8.15,2.25){17}
\put(9.15,2.25){18}
\put(0.3,1.25){4}
\put(1.3,1.25){5}
\put(2.3,1.25){7}
\put(3.15,1.25){10}
\put(0.15,0.25){13}
\put(1.15,0.25){14}
\put(2.15,0.25){15}
\put(3.15,0.25){16}
\end{picture}
\]
where the descents are shown with bars and the boxes of $\lambda/\mu$ are dashed.  Since 15 is not a descent in this SYT, we apply a permutation to 15, 16 and 17 to get
\setlength{\unitlength}{5mm}
\[
\begin{picture}(10,3)(0,0)
\put(0,3){\line(1,0){8}}
\put(0,2){\line(1,0){8}}
\put(0,1){\line(1,0){4}}
\put(0,0){\line(1,0){3}}
\multiput(0,3)(1,0){4}{\line(0,-1){3}}
\put(4,3){\line(0,-1){2}}
\multiput(5,3)(1,0){4}{\line(0,-1){1}}
\multiput(3.1,0)(0.2,0){5}{\line(1,0){0.1}}
\multiput(8.1,2)(0.2,0){10}{\line(1,0){0.1}}
\multiput(8.1,3)(0.2,0){10}{\line(1,0){0.1}}
\multiput(4,0.9)(0,-0.2){5}{\line(0,-1){0.1}}
\multiput(9,2.9)(0,-0.2){5}{\line(0,-1){0.1}}
\multiput(10,2.9)(0,-0.2){5}{\line(0,-1){0.1}}
\put(0.3,2.25){1}
\put(1.3,2.25){2}
\put(2.3,2.25){$\overline{3}$}
\put(3.3,2.25){$\overline{6}$}
\put(4.3,2.25){8}
\put(5.3,2.25){$\overline{9}$}
\put(6.15,2.25){11}
\put(7.15,2.25){$\overline{12}$}
\put(8.15,2.25){$\overline{15}$}
\put(9.15,2.25){18}
\put(0.3,1.25){4}
\put(1.3,1.25){5}
\put(2.3,1.25){7}
\put(3.15,1.25){10}
\put(0.15,0.25){13}
\put(1.15,0.25){14}
\put(2.15,0.25){16}
\put(3.15,0.25){17}
\put(10.3,0){{,}}
\end{picture}
\]
which has the required descent set.

The more complicated case will be when $\lambda$ is a rectangle.  Suppose $\lambda=99$, in which case our construction could give
\setlength{\unitlength}{5mm}
\[
\begin{picture}(8,2)(0,0)
\put(0,2){\line(1,0){9}}
\put(0,1){\line(1,0){9}}
\put(0,0){\line(1,0){6}}
\multiput(0,2)(1,0){7}{\line(0,-1){2}}
\multiput(7,2)(1,0){3}{\line(0,-1){1}}
\multiput(6.1,0)(0.2,0){15}{\line(1,0){0.1}}
\multiput(7,0.9)(0,-0.2){5}{\line(0,-1){0.1}}
\multiput(8,0.9)(0,-0.2){5}{\line(0,-1){0.1}}
\multiput(9,0.9)(0,-0.2){5}{\line(0,-1){0.1}}
\put(0.3,1.25){1}
\put(1.3,1.25){2}
\put(2.3,1.25){$\overline{3}$}
\put(4.3,1.25){$\overline{6}$}
\put(5.3,1.25){8}
\put(6.3,1.25){$\overline{9}$}
\put(7.15,1.25){11}
\put(8.15,1.25){$\overline{12}$}
\put(5.15,0.25){15}
\put(8.15,0.25){18}
\put(0.3,0.25){4}
\put(3.3,1.25){5}
\put(1.3,0.25){7}
\put(2.15,0.25){10}
\put(3.15,0.25){13}
\put(4.15,0.25){14}
\put(6.15,0.25){16}
\put(7.15,0.25){17}
\end{picture}
\]
where 15 is not a descent.
In this case, we have to apply a permutation to the entries 11--15 to get
\[
\begin{picture}(8,2)(0,0)
\put(0,2){\line(1,0){9}}
\put(0,1){\line(1,0){9}}
\put(0,0){\line(1,0){6}}
\multiput(0,2)(1,0){7}{\line(0,-1){2}}
\multiput(7,2)(1,0){3}{\line(0,-1){1}}
\multiput(6.1,0)(0.2,0){15}{\line(1,0){0.1}}
\multiput(7,0.9)(0,-0.2){5}{\line(0,-1){0.1}}
\multiput(8,0.9)(0,-0.2){5}{\line(0,-1){0.1}}
\multiput(9,0.9)(0,-0.2){5}{\line(0,-1){0.1}}
\put(0.3,1.25){1}
\put(1.3,1.25){2}
\put(2.3,1.25){$\overline{3}$}
\put(4.3,1.25){$\overline{6}$}
\put(5.3,1.25){8}
\put(6.3,1.25){$\overline{9}$}
\put(3.15,0.25){11}
\put(7.15,1.25){$\overline{12}$}
\put(8.15,1.25){$\overline{15}$}
\put(8.15,0.25){18}
\put(0.3,0.25){4}
\put(3.3,1.25){5}
\put(1.3,0.25){7}
\put(2.15,0.25){10}
\put(4.15,0.25){13}
\put(5.15,0.25){14}
\put(6.15,0.25){16}
\put(7.15,0.25){17}
\end{picture}
\]
with the required descent set.
\end{example}

Returning to the proof, if $t_1, t_2, \ldots, t_k$ are entries of $T$, let $(t_1, t_2, \ldots, t_k)$ denote the permutation of the entries of $T$ that behaves like the usual cycle notation: $t_1$ is sent the the box containing $t_2$, while $t_2$ is sent to the box containing $t_3$, and so on, with $t_k$ sent to the box containing $t_1$.  

\subsection*{Correcting the filling: the non-rectangle case}   The easier case is when $\lambda$ is not a rectangle.  Suppose first that $N$ is on a higher row than $N_{l-1}+1$.  Let $N'$ denote the lowest number greater than $N_{l-1}+1$ that is on a higher row than $N_{l-1}+1$, and suppose $N'$ is on row $k$.  Apply the cycle 
\[
(N', N'-1, \ldots, N_{l-1}+1, N_{l-1})
\]
to the entries of $T$ to get a tableau $T'$.  In Example~\ref{permuteentries:exa}, this cycle is $(17,16,15)$. Since $N_{l-1}$ is not a descent in $T$ and $N_{l-1}+1, N_{l-1}+2, \ldots, N$ form a horizontal strip, the entries of the cycle form a horizontal strip.  Every entry of this horizontal strip has no box below it and, except in row $k$, these entries appear at the right end of their rows, possibly with other entries of the cycle.  Each cycle entry, except in row $k$, is replaced by the number which is one bigger than it.  As a result, $T'$ will be an SYT, except possibly because of inequalities violated by $N_{l-1}$.  A problem would be created at $N_{l-1}$ in $T'$ if and only if some number strictly between $N_{l-1}$ and $N'$ were immediately to the left or above $N'$ in $T$.  However, by definition of $N'$, no such number exists.  Thus $T'$ is an SYT.  All entries of the cycle except $N_{l-1}$ maintain their relative left-to-right ordering, and $N'$ will be strictly lower in $T'$ than in $T$.  Thus each entry of the cycle except possibly $N_{l-1}$ will maintain its property of not being a descent.  Moreover, $N_{l-1}-1$ will still not be a descent since $N_{l-1}$ has moved to a higher row while $N_{l-1}-1$ has stayed in the same place.  
Since $N_{l-1}$ was not a descent in $T$, $N'$ is on a higher row than both $N_{l-1}+1$ and $N_{l-1}$ in $T$.  Thus $N_{l-1}$ is on a higher row than $N_{l-1}+1$ in $T'$, and hence $N_{l-1}$ is a descent in $T'$, as required.  

We next suppose that $N$ is not on a higher row than $N_{l-1}+1$.  In this case, we will show that $\lambda$ is a rectangle, and consider the rectangle case in the subsequent portion of the proof.  By definition of $\lambda/\mu$ and the filling $T$, since these entries $N$ and $N_{l-1}+1$ are both in $\lambda/\mu$, they must both be on the bottom row.  Thus the entries $N_{l-1}+1, N_{l-1}+2, \ldots, N$ appear in order at the right end of the bottom row of $\lambda$.  If $r$ (from Figure~\ref{fig:mainproof}) equals $\alpha_l$, then these entries $N_{l-1}+1, N_{l-1}+2, \ldots, N$ will completely fill the bottom row of $\lambda$.  This contradicts our assumption that $N_{l-1}$ is not a descent in $T$.  Since we know that $r \leq \alpha_l$, we deduce that $r<\alpha_l$, and so $r=0$ for $N_{l-1}+1, N_{l-1}+2,\ldots, N$ to all be in the bottom row of $\lambda$. Then our construction of $\lambda/\mu$ implies that it includes a box from the rightmost column of $\lambda$.  By the definition of $T$, we conclude that the entry $N$ appears in the rightmost column of $\lambda$.  Thus $N$ is both in the rightmost column and bottom row of $\lambda$, implying that $\lambda$ is a rectangle, the case we consider next.  

\subsection*{Correcting the filling: the rectangle case}  Suppose $\lambda = (m^k)$.  Since $\lambda/\mu$ is a horizontal strip, the entries $N_{l-1}+1, N_{l-1}+2, \ldots, N$ appear in order at the right end of the bottom row of $\lambda$.  Since $N_{l-1}$ is not a descent and the   highest descent is $N_{l-2}$, the entries $N_{l-2}+1, N_{l-2}+2, \ldots, N_{l-1}$ all appear in order on the bottom row of $\lambda$, next to $N_{l-1}+1$.  Since $N_j$ is a descent for $j \leq l-2$, it appears above the bottom row of $\lambda$.  Find the largest entry not of the form $N_j$ that appears above the bottom row.  Such an entry must exist since $\lambda$ is a partition, since~\eqref{equ:rowlengths} holds, and since all $\alpha_l$ entries from $N_{l-1}+1$ to $N$ went in the bottom row of $\lambda$. Because entries $N_{j-1}+1, N_{j-1}+2, \ldots, N_j$ form a horizontal strip in $T$, the largest entry not of the form $N_j$ that appears above the bottom row must be $N_j-1$, for some $j \leq l-2$.  Fixing this latter $j$, suppose $N_j-i$ is the smallest entry greater than $N_{j-1}$ that also appears above the bottom row, implying that $N_j-i, N_j-i+1, \ldots, N_j-1, N_j$ all appear above the bottom row.  We can now specify the permutation to apply to the entries of $T$, namely
\begin{multline*}
(N_j+1, N_j+2, \ldots, \widehat{N_{j+1}}, \ldots, \widehat{N_{j+2}}, \ldots, \widehat{N_{l-2}}, N_{l-2}+1, N_{l-2}+2, \ldots, N_{l-1},\\ N_{l-2}, N_{l-3}, \ldots, N_{j+1}, N_j, N_j-1, N_j-2, \ldots, N_j-i)
\end{multline*}
where, as usual, a hat denotes omission.  In words, the cycle starts at $N_j+1$ and goes all the way along the bottom row $k$ of $\lambda$ as far as $N_{l-1}$, then jumps up to $N_{l-2}$ which, because of its size, is the rightmost entry on row $k-1$.  It stays above the bottom row, picking up the biggest entries, namely those of the form $N_n$ for $j \leq n \leq l-2$, followed by the important horizontal strip $N_j-1, N_j-2 \ldots, N_j-i$. The cycle returns back to the bottom row at $N_j+1$.  In Example~\ref{permuteentries:exa}, this cycle is $(13,14,15,12,11)$.
As before, let $T'$ denote the image of $T$ under this permutation.

By the definition of $N_j-i$, the bottom row of $T'$ will be   increasing.  Restricting our attention to the entries above the bottom row, we get that each is replaced by the next biggest entry or stays the same, implying that the relevant inequalities will be preserved.  Note that $N_{l-1}$ in $T'$ is at the rightmost box in row $k-1$, and so does not violate any inequalities.  It remains to show that the inequalities between the bottom row and the row above it are valid in $T'$.  We will do this by showing, roughly, that the elements above the bottom row  have not moved ``too far to the left.''  More precisely, to show that $T'$ is an SYT, it is now sufficient to show every element in row $k-1$ or above is less than the element of the bottom row directly below it.  Moreover, it suffices to show this for entries weakly between $N_j-i$ and $N_{l-1}$, since these are exactly the ones that change places.  First, we have that $N_{l-1}$ is above $N$.  Then by~\eqref{equ:rowlengths}, $N_{l-2}$ is weakly to the right of $N_{l-1}+1$, and so $N_{l-2}$ is less than the element of the bottom row directly below it.  Continuing in this fashion, we have $N_{l-3}, N_{l-4}, \ldots, N_j$ are weakly to the right of $N_{l-2}+1, N_{l-3}+1, \ldots, N_{j+1}+1$, respectively, implying the necessary inequalities.  Next, $N_j-1$ is weakly right of $N_{j+1}-1$, and $N_j-2$ is weakly to the right of $N_{j+1}-2$, all the way up to $N_j-(i-2)$ being weakly to the right of $N_{j+1}-(i-2)$.  Finally, we need that $N_j-(i-1)$ is strictly to the right of $N_j-i$ in $T'$.  To do this, we will show the equivalent fact that, in $T$, $N_j-i$ is strictly to the right of $N_j+1$.  

The situation for $T$ that puts $N_j-i$ as far left as possible is when row $k-1$ of $\lambda$ ends on the right with the sequence
\[
N_j-i, N_j-i+1, \ldots, N_j-2, N_j-1, N_j, N_{j+1}, \ldots, N_{l-2}.
\]
There are $i+(l-2)-j+1 = (l-j)+i-1$ elements in this sequence.  On the other hand, in the bottom row, counting the number of elements in the sequence that starts at $N_j+1$ and goes right to the end of the row, we get $N-(N_j+1)+1-((l-2)-(j+1)+1)$ elements.  This is because all numbers weakly between $N_j+1$ and $N$ are included, except for $N_{j+1}, N_{j+2}, \ldots, N_{l-2}$.  This count of elements in the bottom row simplifies to $N-N_j-(l-j)+2$.  Thus to show that $N_j-i$ is strictly to the right of $N_j+1$ in $T$, we need  
\[
(l-j)+i-1 < N-N_j-(l-j)+2,
\]
or equivalently,
\begin{equation}\label{equ:finalineq}
2(l-j)+i-3 < \alpha_{j+1} + \alpha_{j+2} + \ldots + \alpha_l.
\end{equation}
Since $N_j-i > N_{j-1}$, we know that $ i \leq \alpha_j-1$, and so $ i \leq a$.  Meanwhile, the right-hand side of \eqref{equ:finalineq} is at least $(l-j)a$.  So it suffices to have
\[
2(l-j)+a-3 < (l-j)a,
\] 
which can be written as
\begin{equation}\label{equ:onelastineq}
(l-j-1)(a-2) > -1.
\end{equation}
Now $j\leq l-2$ so $l-j-1\geq 1$.  Therefore, \eqref{equ:onelastineq} holds so long as $a>1$.  If $a=1$, then by \eqref{equ:rowlengths}, the right-hand side of \eqref{equ:finalineq} is at least $(l-j)2$.  So it suffices to have 
\[
2(l-j)+a-3 < 2(l-j),
\]
which is true since $a=1$. We conclude that $T'$ is an SYT.

To work out the descent set of $T'$, the first thing to notice is that the entries $N_{l-1}+1$ and above have remained at the right end of the bottom row, so none of them are descents.  As required, $N_{l-1}$ is a new descent in $T'$ because it appears in row $k-1$.  Since $N_{l-2}, N_{l-3}, \ldots, N_{j+1}, N_j$ remain above the bottom row, they are still descents.  Also, $N_{j-1}$ is still a descent: the only thing to check is the case when $N_{j-1}+1 = N_j-i$, but since $N_j-i$ is moved lower while $N_{j-1}$ remains in place, this descent will be preserved.  All the entries on the bottom row will obviously continue to not be descents.  Therefore, it remains to check that $N_j-1, N_j-2, N_j-i$ are not descents in $T'$.  Since $N_j-2, N_j-3, \ldots, N_j-i$ are not descents in $T$, 
under the cycle, this translates to $N_j-1, N_j-2, \ldots, N_j-i+1$ not being descents in $T'$.  Finally, $N_j-i$ is in the bottom row in $T'$ and so is certainly not a descent.  We conclude that $T'$ is an SYT of descent set $\{N_1, N_2, \ldots, N_{l-1}\} = S(\alpha)$.
\end{proof}

\section{Concluding remarks and open problems}\label{sec:conclusion}

\subsection{Christoffel words}\label{sub:christoffel}
The most obvious open problem is Question~\ref{que:schurpos}, which Conjecture~\ref{con:max} aims to resolve.  Referring to the illustration on the left in Figure~\ref{fig:billiard} as an example, we note that the bold lines corresponding to portion of the border of the ribbon above the diagonal give a Christoffel word (see \cite{BLRS09} and the references therein) when read from right to left.  For the illustration on the right in the same figure, the number of rows is not coprime to the number of columns, so a Christoffel word does not result; instead, the resulting word is a concatenation of three Christoffel words.  One wonders if the theory of combinatorics on words could help resolve Conjecture~\ref{con:max}, perhaps by restating Conjecture~\ref{con:max}(b) in a more accessible way.  

\subsection{Skew shapes with full support}
Theorem~\ref{thm:fullsupport} shows that the support of equitable ribbons $R$ is as big as possible.  More precisely, $\lambda \in \supp{R}$ if and only if $|\lambda|=|R|$ and $\rows{R} \domleq \lambda \domleq \cols{R}^t$.  Another way to think of this is in terms of the dominance lattice for partitions of size $N$: the support of $R$ is the full interval $[\rows{R}, \cols{R}^t]$ in the dominance lattice.

\begin{question}\label{que:fullsupport}
What other skew shapes $A$ have ``full'' support, meaning that $\supp{A}$ is the entire interval $[\rows{A}, \cols{A}^t]$ in the dominance lattice of partitions of size $N$?
\end{question}

The answer seems to be not at all obvious.  For example, the ribbons $\rib{442}$, $\rib{424}$ and $\rib{242}$ all have full support, but $\rib{422}$ does not.  Olga Azenhas, Alessandro Conflitti and Ricardo Mamede \cite{ACM10pr} 
have recently answered Question~\ref{que:fullsupport} in the case when $A$ is \emph{multiplicity-free}, i.e., when $s_A$ is expanded in the basis of Schur functions, all the coefficients are 0 or 1.  

\subsection{The equitable ribbons in $\bigP{N}$}
We know from Corollary~\ref{cor:full_reduction} that answering Question~\ref{que:schurpos} amounts to finding those equitable ribbons that are maximal in $\bigP{N}$.  Since the equitable ribbons play an even more central role in the support case, one might study the subposet of $\bigP{N}$ consisting of the equitable ribbons for its own interest.  A natural question to ask is which equitable ribbons are minimal in  this subposet.

As before, let us fix $N$ and the number of rows $l$, implying that there are $N-l+1$ columns.  If $N$ and $l$ are such that the equitable ribbons will have both a row and column of length 1, this case has already been solved in Example~\ref{exa:lengthone}.  Switching the roles of $l$ and $N-l+1$ if necessary---in effect transposing the ribbons in question---we can therefore restrict our attention to equitable ribbons whose rows all have length at least 2.  We conjecture that the minimal equitable ribbons are then exactly those of the form 
\[
\rib{\underbrace{a+1, a+1, \ldots, a+1}_{r\ \mathrm{copies}}, a, a, \ldots, a, \underbrace{a+1, a+1, \ldots, a+1}_{r'\ \mathrm{copies}}}
\]
where $|r-r'|\leq 1$. Observe that $a=\lfloor\frac{N}{l}\rfloor$ so, for fixed $N$ and $l$, there is exactly one such element up to antipodal rotation.  This conjecture is a special case of the following one.

\begin{conjecture}\label{con:minrib}
Consider the subposet of $\bigP{N}$ consisting of those ribbons $R$ such that $\rows{R} = \lambda$ for some fixed $\lambda \vdash N$.  This subposet has a unique minimal element, namely $[R]$ with 
\[
R = \rib{\lambda_1, \lambda_3, \lambda_5, \ldots, \lambda_{\ell(\lambda)}, \ldots, \lambda_6, \lambda_4, \lambda_2}.
\]
\end{conjecture}
 Conjecturing a \emph{maximal} element for the subposet of Conjecture~\ref{con:minrib} with general $\lambda$ is a more difficult task.  

\section*{Acknowledgments} We thank Pavlo Pylyavskyy for interesting discussions, including those that resulted in the formulation of Conjecture~\ref{con:max}.
 We are grateful to the anonymous referee for helpful comments.
The Littlewood--Richardson calculator \cite{BucSoftware} and the posets package \cite{SteSoftware} were used for data generation.

\section*{References}

\bibliography{maxsupport}
\bibliographystyle{alpha}

\end{document}